\documentclass[a4paper,fleqn]{cas-sc}

\usepackage[authoryear,longnamesfirst]{natbib}
\usepackage{amsmath}
\usepackage{amsthm}
\def\tsc#1{\csdef{#1}{\textsc{\lowercase{#1}}\xspace}}
\tsc{WGM}
\tsc{QE}
  
\def\TM{T_{max}} 
\newtheorem{theorem}{Theorem}[section]

\newtheorem{lemma}[theorem]{Lemma}

\begin{document}
\let\WriteBookmarks\relax
\def\floatpagepagefraction{1}
\def\textpagefraction{.001}

\shorttitle{Global solvability of a Gompertz chemotaxis model}    

\shortauthors{}  

\title [mode = title]{Uniform boundedness for the two-dimensional Keller--Segel system with Gompertz growth}  

\tnotemark[1] 

\tnotetext[1]{} 

%

\author[1]{Nohayla ALAOUI }






\affiliation[1]{organization={Faculty of Sciences Semlalia, Cadi Ayyad University},
            addressline={Laboratory of Mathematics, Modeling and Automatic Systems}, 
            city={Marrakech},
            postcode={Boulevard Prince Moulay Abdellah, B.P. 2390}, 
            country={Morocco}}

\author[1]{Mohamed HALLOUMI}





\author[2]{Giuseppe VIGLIALORO}
\cormark[1]





\affiliation[2]{organization={Università degli Studi di Cagliari},
            addressline={Dipartimento di Matematica e Informatica}, 
            city={Cagliari},
            postcode={Via Ospedale, 72,
09124},
            country={Italy}}

\cortext[1]{giuseppe.viglialoro@unica.it}



\begin{abstract}
It is well known that in two dimensions the classical Keller--Segel model may exhibit cell aggregation phenomena, which are automatically controlled when a logistic source with quadratic degradation is introduced. Considerable effort has been devoted to identifying weaker damping mechanisms that still reproduce the stabilizing effect of the quadratic term. In \cite{xiang2018sub}, it was shown that, under suitable assumptions on the initial distribution of the cell density, even sub-logistic sources of the form $
f(s) = a s - \frac{b s^2}{\ln(\ln (s + e))}$, $a,b>0$, and $s>0$,
can achieve this effect.
In the present paper, we introduce the Gompertz function
$
f(s) = \alpha s \ln\!\left(\frac{K}{s}\right)$, $\alpha,K>0$, and $s>0$, as an external source in minimal 2D chemotaxis model. This term exhibits a weaker damping effect than the prototype considered in \cite{xiang2018sub}. We investigate its influence on the system and establish conditions ensuring global existence and uniform boundedness of solutions.
\end{abstract}



\begin{keywords}
Chemotaxis model \sep Gompertz growth source \sep Boundedness\sep
\end{keywords}

\maketitle

\section{Introduction, motivation and presentation of the main result}
Chemotaxis refers to the directed movement of cells or organisms in response to chemical gradients. This process is essential in numerous biological systems. In addition, chemotaxis is a key factor in cancer cell migration and invasion. In mathematical biology, this phenomenon is generally categorized according to the characteristics of the cellular response to the extracellular signal.

The classical Keller--Segel chemotaxis model (\cite{keller1970initiation, keller1971model}), originally proposed in the 1970s to describe the directed movement of mobile cells, considers the cell density 
\(u = u(x,t)\), where \(x\) is the spatial variable and \(t\) represents time, moving toward a self-produced chemical signal \(v = v(x,t)\). 
The model is well-known for its ability to capture aggregation dynamics through the cross-diffusion term 
\(- \nabla \cdot (u \nabla v)\). The related formulation  is given by
\begin{equation}\label{Model-KS}
    \begin{cases}
        u_t = \Delta u - \chi \nabla \cdot (u\nabla v), & \textrm{on  }\; \Omega \times (0,T_{\max}),\\[4pt]
        \tau v_t = \Delta v - v + u, & \textrm{on  }\; \Omega \times (0,T_{\max}),\\[4pt]
        \textrm{$+$ boundary conditions (BC) $+$ initial data (ID)},
    \end{cases}
\end{equation}
where \(\Omega \subset \mathbb{R}^n\) ($n\geq 1$) is a bounded smooth domain, \(\chi >0\), \(\tau \in \{0,1\}\), and  $T_{\max}\in (0,\infty]$ indicates the upper bound of the time interval during which the system evolves.

The chemotactic aggregation mechanism is most transparent in the parabolic--elliptic case ($\tau=0$). 
Substituting the elliptic relation $\Delta v = v - u$ into the equation for $u$ shows heuristically that 
the cross-diffusion behaves like a quadratic term $\chi u^{2}$.  This aggregation may lead to finite- or infinite-time blow-up depending on the spatial dimension. 
In one dimension, for both values of $\tau$, no concentration occurs independently by $\chi$ and the ID, whereas in two or higher dimensions blow-up may arise, for some ID and $\chi$.  In particular, in the two-dimensional parabolic--elliptic setting, the critical mass phenomenon determines the threshold between global existence and blow-up. (The literature is vast, and more detailed references can be found in this survey \cite{BellomoEtAl}.)

Exactly to reduce/prevent this over-aggregation, and obtain more biologically realistic population dynamics, the \(u\)-equation was modified by researchers \cite{osaki2002exponential, winkler2010boundedness, zheng2018new} who proposed to incorporate a logistic source of the form \(au - bu^{2}, a, b>0\). More precisely, even an arbitrarily small logistic damping coefficient \(b > 0\) is sufficient to preclude any blow-up entirely in two dimensions, ensuring that all solutions are global in time and uniformly bounded for both the parabolic-elliptic (\(\tau = 0\)) and the fully parabolic (\(\tau =1\)) cases; see, respectively, \cite{TelloWinkParEl,osaki2002exponential}. 

Since this work is devoted to the analysis of system \eqref{Model-KS} with non-logistic source terms (i.e., not of the form $au - bu^2$), we shall highlight the features that distinguish the source term $f$ considered here from those commonly treated in the literature.  More precisely, to give a complete formulation of the problem, we now explicitly specify the boundary conditions and the initial data. So, introducing the notation $\nu$ for the outward unit normal vector on $\partial \Omega$, the general chemotaxis--growth system can be written as follows:
\begin{equation}\label{eq:chem-gompertz}
    \begin{cases}
        u_t = \Delta u - \chi \nabla \cdot (u\nabla v) + f(u), & \textrm{on  }\; \Omega \times (0,T_{\max}),\\[4pt]
        \tau v_t = \Delta v - v + u, & \textrm{on  }\; \Omega \times (0,T_{\max}),\\[4pt]
        \dfrac{\partial u}{\partial \nu} = \dfrac{\partial v}{\partial \nu} = 0, & \textrm{in  }\; \partial \Omega \times (0,T_{\max}),\\[4pt]
        u(x,0)=u_0(x)\geq 0,\;\tau v(x,0)=\tau v_0(x)\geq 0, & x\in\bar{\Omega}.
    \end{cases}
\end{equation}
The objective of this work is to analyze problem \eqref{eq:chem-gompertz} when the function $f$ is of Gompertz-type form, i.e. \(
f(u) = \alpha u \ln\left(\frac{K}{u}\right),
\)
with \(\alpha>0\), representing the intrinsic growth rate and \(K > 0\) the carrying capacity. From the biological perspective,  the Gompertz growth function, which is well-documented in the literature on cancer biology literature (\cite{vaghi2020population}), provides a more accurate representation of tumor cell proliferation than logistic growth. In contrast to the logistic model, which assumes symmetric growth around an inflection point, the Gompertz model exhibits an asymmetric exponential decay of the growth rate, better reflecting the deceleration of tumor growth as it approaches its carrying capacity. All these features of the Gompertz curve make it more useful to fit data on some growth processes (see, for instance, \cite{dhar2018comparison}). If we further add that, in the context of chemotaxis models, the literature that addresses this analysis remains, to our knowledge, rather limited (see items \ref{Item1Winkler} and \ref{Item2Xiang}, below), we believe that the present investigation is naturally motivated and can contribute to further developments in the field.
\begin{enumerate}[I.]
\item \label{Item1Winkler} For a set of functions for which the Gompertz function is an element, in \cite{Winkler2011-BlowUHighDimensionLogistic} it is shown that for a close model  to \eqref{eq:chem-gompertz}, when $n\geq 5$ 
 there exist initial data such that the corresponding solution blows up in finite time. 
\item \label{Item2Xiang} In \cite{xiang2018sub} for a set of sub-logistic sources with a kinetic prototype $a s - \frac{b s^2}{\ln(\ln(s + e)}$, $a,b>0$, and $s> 0$, it is established that in two dimensions all solutions are uniformly bounded under some restrictions connecting $\chi$, the initial mass $\int_\Omega u_0(x)dx$ and some asymptotic aspects of $f(u)$. 
\end{enumerate}
Specifically, if $f_{\text{G}}$ denotes the Gompertz function, $f_{\text{L}}$ the classical logistic function, and $f_{\text{S}}$ the logistic-type function introduced in item \ref{Item2Xiang}, one can observe the order $
f_{\text{G}} < f_{\text{S}} \lesssim f_{\text{L}},$
which indicates that $f_{\text{G}}$ has a very weakly limited damping (i.e., mortality is barely constrained) for large $u$, $f_{\text{S}}$ produces an intermediate damping effect, and $f_{\text{L}}$ exerts the strongest damping.

In particular, as the Gompertz term does not belong to the class of functions considered in \cite{xiang2018sub} (see \cite[Remark 1.2]{xiang2018sub}), it appears natural to devote attention to studying the boundedness issue in model \eqref{eq:chem-gompertz} for external reactions that behave like $f_{\text{G}}$, i.e., for Gompertz sources. 

This is precisely what we aim to do in this paper, and in order to present our result, we assume that:
\begin{equation}\label{AssumptionDomain}
\begin{cases}
\textrm{for some } \delta \in (0,1), \; \Omega \subset \mathbb{R}^2 \textrm{ is a bounded domain of class } C^{2+\delta}, \\
\textrm{for some } \alpha,K>0, f(s)=\alpha s \ln\left(\frac{K}{s}\right), s>0, \\
\chi>0, \tau\in\{0,1\} \textrm{ and } 0<u_{0},\; 0 \leq  \tau v_{0} \in C^{2+\delta}(\overline{\Omega}),
\quad \text{ with }
\partial_{\nu} u_{0} = \partial_{\nu} (\tau v_{0}) = 0\quad \text{on } \partial\Omega. \\
\end{cases}
\end{equation}
We will establish this
\begin{theorem}\label{MainTheorem}
Under assumptions \eqref{AssumptionDomain}, there exists a positive constant $C_{GN}$ such that whenever 
\begin{equation}\label{AssumptionsTheorem}
 K>e^{-\frac{2}{\alpha}} \quad \textrm{and} \quad \chi M\leq \frac{1}{2 C_{GN}^4},    
\end{equation}
being $M:=\max\left\{\int_\Omega u_0(x)dx,K|\Omega|\right\}$, it is possible to find a uniquely determined pair of functions $\
(u,v) \in 
C^{2+\delta,\,1+\frac{\delta}{2}}(\overline{\Omega} \times [0,\infty)) 
\times 
C^{2+\delta,\,\tau+\frac{\delta}{2}}(\overline{\Omega} \times [0,\infty)),$
solving problem \eqref{eq:chem-gompertz} and such that $0<u,v \in L^\infty(\Omega \times (0,\infty)).$
\end{theorem}
\section{Local existence and boundedness criterion}
The first result deals with the local-well-posedness of model \eqref{eq:chem-gompertz}; a crucial properties is played by the  extension criterion, in which it remains emphasized the role of the logistic used in this investigation and thanks to which a barrier method can be performed. 
\begin{lemma}\label{LocalExistenceLemma}
Let assumptions \eqref{AssumptionDomain} be fulfilled. 
Then there exist
$T_{\max} \in (0,\infty]$ and a uniquely determined pair of positive functions
\[
(u,v) \in 
C^{2+\delta,\,1+\frac{\delta}{2}}(\overline{\Omega} \times [0,T_{\max})) 
\times 
C^{2+\delta,\,\tau+\frac{\delta}{2}}(\overline{\Omega} \times [0,T_{\max})),
\]
solving problem \eqref{eq:chem-gompertz}, and such that
\begin{equation}\label{DichotomyCriterion}
 \text{ either }  T_{\max}=\infty,  \text{ or } \limsup_{t \to T_{\max}} \|u(\cdot,t)\|_{L^\infty(\Omega)} = \infty.
\end{equation}
\begin{proof}
By consistently defining $f(0)=0$, the existence of a short time $0<T=T(u_0,\tau v_0)<1$ and of a classical solution $(u,v)$ follow from standard fixed point procedures (see, for instance, \cite{BaghaeiEtAl2026JDE} for details developed  in the same regularity setting herein posed).
Additionally, by prolonging the solution beyond $T$, one can find $\TM\in (0,\infty]$ with the property that if $T_{\max} < \infty$
\begin{equation}\label{DichotomyCriterionBIS}
 \text{ then either }
\limsup_{t \to T_{\max}} \|u(\cdot,t)\|_{L^\infty(\Omega)} = \infty
\quad \text{or} \quad
\liminf_{t \to T_{\max}} \inf_{x \in \Omega} u(x,t) = 0.
\end{equation}  
Indeed, if for some finite $T_{\max}$ there were constants $C>0$ and $c>0$ such that $$
\sup_{t<T_{\max}} \|u(\cdot,t)\|_{L^\infty(\Omega)} \le C
\quad\text{and}\quad
\inf_{t<T_{\max}} \inf_{x\in\Omega} u(x,t) \ge c > 0,$$
from the one hand the function $
f(u)=\alpha u\ln(K/u)$ would be smooth and globally Lipschitz on $[c,C]$, and from the other (by using in the equation for $v$ classical elliptic and parabolic regularity results as those in \cite{BrezisBook,LSUBookInequality}), one would obtain that $
\sup_{t<T_{\max}} \|\nabla v(\cdot,t)\|_{C^{\delta_1}(\bar{\Omega})} < \infty$, for some $0<\delta_1<1$. In turn, exploiting this information in the equation for $u$,  would yield $
\sup_{t<T_{\max}}
\|u(\cdot,t)\|_{C^{2+\delta_1}(\overline{\Omega})}
< \infty$, this allowing to consider $u(\cdot,T_{\max})$ as an admissible initial data. As a consequence, a classical solution on
$(T_{\max},T_{\max}+\rho)$ for some $\rho>0$, could be obtained, contradicting the maximality of $T_{\max}$. At this stage, noting that $\sup_{t<T_{\max}}
\|u(\cdot,t)\|_{C^{2+\delta_1}(\overline{\Omega})}
< \infty$ allows to find $c_v = \|\Delta v(\cdot,t)\|_{L^\infty(\Omega)}$ on $(0,T_{\max})$, the dichotomy criterion \eqref{DichotomyCriterionBIS} can be actually refined. If, in fact, we assume by contradiction that $T_{\max} < \infty$ and
\[
\limsup_{t \to T_{\max}} \|u(\cdot,t)\|_{L^\infty(\Omega)} < \infty,
\]
since $-\alpha \ln\left(\frac{K}{s}\right) \searrow -\infty$ as $s \to 0^+$,  any constant function $0<\underline{u}=c_* \leq K e^\frac{-\chi c_v}{\alpha }$ is such that 
\[
\underline{u}_t - \Delta \underline{u} + \chi \nabla \cdot (\underline{u} \nabla v) - \alpha \, \underline{u} \ln\left( \dfrac{K}{\underline{u}}\right)\leq  \underline{u}\chi\|\Delta v(\cdot,t)\|_{L^\infty(\Omega)}-\alpha \underline{u} \ln\left(\frac{K}{\underline{u}}\right)\leq \underline{u}\left(\chi c_v-\alpha \ln\left(\frac{K}{\underline{u}}\right)\right) \leq 0 \quad \textrm{on } (0,T_{\max}). 
\]
Being $\underline{u}$ is a subsolution to the first equation in \eqref{eq:chem-gompertz}, a parabolic comparison principle implies
\[
u(x,t) \ge \min\left\{c_*,\min_{x\in \bar{\Omega}}u_0(x)\right\}>0 \quad \textrm{for all } (x,t)\in \Omega \times (0,T_{\max}),
\]
so that the second alternative in \eqref{DichotomyCriterionBIS} is ruled out and \eqref{DichotomyCriterion} is established.
%
%


%
%
%
%
Finally, the parabolic and elliptic strong maximum principles imply as well $v>0$ throughout $\overline{\Omega} \times (0,T_{\max})$, whereas the uniqueness of the solution can be established following arguments analogous to the aforementioned \cite{BaghaeiEtAl2026JDE}.
\end{proof}
\end{lemma}  
Indicating, from now on, with $(u,v)$ the local solution to model \eqref{eq:chem-gompertz} defined in $\Omega \times (0,T_{\max})$ and provided by Lemma \ref{LocalExistenceLemma}, we also will rely on this boundedness criterion:
\begin{lemma}\label{LemmaBoundendesCriterion} 
If $(u,v)$ satisfies
\begin{equation}\label{ApripriBound}
\int_\Omega u \ln (u)\le L \quad \text{and} \quad
\int_\Omega |\nabla v|^2 \le L \quad \text{for all } t \in (0,T_{\max})
\end{equation}
with a positive constant $L$, then there exists $H>0$ such that
\[
\|u(\cdot,t)\|_{L^\infty(\Omega)}  \le H
\quad \text{for all } t \in (0,\infty).
\]
\begin{proof}
The proof is a consequence of \cite[Lemma 3.3.]{BellomoEtAl} (which is based on \cite[Lemma 3.2.]{BellomoEtAl}), and the dichotomy criterion \eqref{DichotomyCriterion}.
\end{proof}
\end{lemma}
As a consequence of this result, global boundedness is achieved once the a priori estimates
in \eqref{ApripriBound} are obtained.
\section{A priori estimates and proof of Thorem \ref{MainTheorem}}
We will rely on these properties of $u$.
\begin{lemma}\label{LemmaMass}
The $u$-component of $(u,v)$ is such that 
\begin{equation}\label{BoundednessMass}
\int_\Omega u\leq M:=\max\left\{\int_\Omega u_0(x)dx,K|\Omega|\right\} \quad \textrm{on } (0,T_{\max}).
\end{equation}
As a consequence, there is $C_{GN}>0$ such that, for some $c_0>0$,
\begin{equation}\label{EstimateintuSquare} 
\int_\Omega u^2\leq 2 C_{GN}^4 M\int_\Omega \dfrac{|\nabla u|^2}{u}+c_0 \quad \textrm{on } (0,T_{\max}).
\end{equation}
\begin{proof}
Integrating \eqref{eq:chem-gompertz} and using the divergence theorem together with homogeneous Neumann boundary conditions, we obtain
\begin{equation}\label{DerivativeMAss}
\frac{d}{dt}\int_\Omega u =\alpha \int_\Omega u \ln \left(\frac{K}{u}\right), \quad   \text{for all } t\in (0,T_{\max}).
\end{equation}
%
%
Letting $\phi(s) := s \ln(K/s)$, for $s>0$, it is seen that 
$\phi$ is concave, so that by Jensen's inequality, we have
\[
\frac{1}{|\Omega|} \int_\Omega \phi(s) ds \le \phi\Big( \frac{1}{|\Omega|} \int_\Omega s  ds \Big).
\]
By applying this estimate to $u=u(x,t)$ leads to
\[
\frac{1}{|\Omega|} \int_\Omega u \ln\left(\frac{K}{u}\right) \le  \frac{1}{|\Omega|} \int_\Omega u \ln\left(\frac{K}{\frac{1}{|\Omega|} \int_\Omega u}\right)\quad \textrm{for all } t\in (0,T_{\max}).
\]
Now we plug the last bound into relation \eqref{DerivativeMAss} so to achieve, for all $t\in (0,T_{\max})$ and for the function $z=z(t):=\int_\Omega u$, the following ODI
\[
z'\leq  \alpha z\ln \left(\frac{K|\Omega|}{z}\right)  \quad \textrm{on } (0,T_{\max}), 
\quad 
z(0)=\int_\Omega u_0(x)dx,
\]
which by comparison provides \eqref{BoundednessMass}.

As to the second claim, by means of the Gagliardo--Nirenberg inequality we can write
\begin{equation}\label{GN_Ineq}
\int_\Omega u^2
=\int_\Omega (\sqrt u)^4=\|\sqrt u\|_{L^4(\Omega)}^4\leq C_{GN}^4\Big[\|\nabla\sqrt u\|^{\frac{1}{2}}_{L^2(\Omega)} \|\sqrt u\|^{\frac{1}{2}}_{L^2(\Omega)}+\|\sqrt u\|_{L^2(\Omega)}\Big]^4\textrm{ for all } t\in (0,T_{\max}).
\end{equation}
Given the elementary relation $(A+B)^q \leq 2^{q-1}(A^q +B^q)$, valid for all nonnegative $A$ and $ B$ and $q\geq 1$, inequalities \eqref{GN_Ineq} and  \eqref{BoundednessMass} provide \eqref{EstimateintuSquare} for some $c_0$.
\end{proof}
\end{lemma}
In order to exploit the boundedness criterion in Lemma \ref{LemmaBoundendesCriterion}, we only need to derive the bounds in \eqref{ApripriBound}. In this direction,  we define 
\begin{equation}\label{DefiFunctionaF}
    F(t)=\int_\Omega u \ln (u) +\tau b\int_\Omega |\nabla v|^2:= I_1(t)+I_2(t)\quad \textrm{for all } t\in (0,T_{\max}),
\end{equation}
with 
\begin{equation}\label{Valueofb}
  b = \frac{\chi}{2}, \quad  \textrm{ implying } 
  -1 + \frac{(\chi + 2 b)^2}{4 b} C_{GN}^4 M=
-1 + 2 \chi M C_{GN}^4, 
\end{equation}
 where the constants $M$ and $C_{GN}^4$ have been introduced in Lemma \ref{LemmaMass}.
Our aim is establishing an absorption inequality of the type
\begin{equation}\label{AbsortionForF}
    F'(t)\leq \delta_1-\delta_2 F(t) \quad \textrm{ on } (0,T_{\max}),\quad \textrm{combined with } 
    F(0)=\int_\Omega u_0(x)\ln (u_0(x))dx+\tau b\int_\Omega |\nabla v_0(x)|^2dx,
\end{equation}
naturally entailing (for $\delta_1,\delta_2>0$)
\begin{equation}\label{BoundednessForF} 
\int_\Omega u \ln (u) +\tau b\int_\Omega |\nabla v|^2=F(t)\le \max\left\{F(0),\,\frac{\delta_1}{\delta_2}\right\} \textrm{ for all } t\in (0,T_{\max}).
\end{equation}
\paragraph{Proof of Theorem \ref{MainTheorem}}  
Let us differentiate with respect to the time the functional $F=F(t)$, introduced in \eqref{DefiFunctionaF}. Some algebraic calculations provide, in conjunction with the divergence theorem,
\begin{equation}\label{MainDerivationF}
    \begin{split}
    F'(t)&=I'_1(t)+I'_2(t)= \int_\Omega u_t \left( \ln (u)  + 1\right) +2 \tau b \int_\Omega \nabla v \cdot \nabla v_t= \int_{\Omega} \left(\Delta u-\chi \nabla \cdot (u \nabla v)+\alpha u \ln \left(\frac{K}{u}\right)\right)\left( \ln (u)  + 1\right)\\
    & \quad +2\tau b \int_{\Omega} \nabla v \cdot \nabla (\Delta v-v+u)\\
    &\leq -\int_\Omega \dfrac{|\nabla u|^2}{u}\!-\!(\chi  +2 b\tau)\int_\Omega u \Delta v+\alpha \ln(K)\int_\Omega u\ln (u) - \alpha \int_\Omega u(\ln (u))^2 +\alpha\int_\Omega u \ln\left(\frac{K}{u}\right)\\
    &\quad -2b\tau\int_\Omega |\Delta v|^2 -2b\tau\int_\Omega |\nabla v|^2 \quad \textrm{on } (0,T_{\max}).
    \end{split}
\end{equation}
By employing Young's inequality for $\tau=1$, and $\chi\leq \dfrac{(\chi  + 2 b)^2}{8 b }$ for $\tau=0$ (in this case using $\Delta v=v-u$), once estimate \eqref{EstimateintuSquare} is invoked, for both $\tau\in\{0,1\}$ and on $(0,T_{\max})$ the following holds true: 
\begin{equation}\label{EstimateLaplacianVu}
   -(\chi  + 2 b\tau)\int_\Omega u \Delta v\leq \dfrac{(\chi  + 2 b)^2}{8 b }\int_\Omega u^2 + 2 b\tau \int_\Omega |\Delta v|^2 
   \leq \dfrac{(\chi  + 2 b)^2}{4 b } C_{GN}^4 M\int_\Omega \dfrac{|\nabla u|^2}{u} + 2 b\tau \int_\Omega |\Delta v|^2 +c_0.  
\end{equation}
On the other hand,  by using Young's inequality and \eqref{BoundednessMass}, and by exploiting now the first assumption on \eqref{AssumptionsTheorem}, we have
\[
- \alpha \int_\Omega u(\ln (u))^2\leq -(\alpha \ln K+2)\int_\Omega u\ln (u) +\dfrac{M (\alpha \ln (K)+2)}{4\alpha} \quad \textrm{for all } t\in (0,T_{\max}), 
\]
which gives for $c_1=\dfrac{M (\alpha \ln (K)+2)}{4\alpha }+\alpha |\Omega|\frac{K}{e}$
\begin{equation}\label{Estimtateulogusquare}
    \alpha \ln(K)\int_\Omega u\ln (u) - \alpha \int_\Omega u(\ln (u))^2 +\alpha\int_\Omega u \ln\left(\frac{K}{u}\right)\leq -2\int_\Omega u\ln (u)+c_1 \quad \textrm{for all } t\in (0,T_{\max}),
\end{equation}
and where we  used $\alpha s \ln\left(\frac{K}{s}\right)\leq \alpha \frac{K}{e}$, for all $s\geq 0.$

At this stage, we plug bounds \eqref{EstimateLaplacianVu} and \eqref{Estimtateulogusquare} into \eqref{MainDerivationF} so that we can write 
\[
F'(t)\leq \left(-1 + \frac{(\chi + 2 b)^2}{4b} C_{GN}^4 M\right) \int_\Omega \dfrac{|\nabla u|^2}{u} +c_0 +c_1-2\left(\int_\Omega u\ln (u)  +b\tau\int_\Omega |\nabla v|^2\right) \quad \textrm{for all } t\in (0,T_{\max}),
\]
and by recalling what analyzed in \eqref{Valueofb} and by 
invoking again the second assumption in \eqref{AssumptionsTheorem},
we arrive at the desired inequality in \eqref{AbsortionForF}.
Let us  distinguish the situations  $\tau=1$ and $\tau=0$.

{\bf Case $\tau=1$: } From estimate \eqref{BoundednessForF}, we entail the uniform-in-time bound of $I_1(t)=\int_\Omega u \ln (u)$ and $I_2(t)=b\int_\Omega |\nabla v|^2$ on $(0,T_{\max})$ so concluding by means of Lemma \ref{LemmaBoundendesCriterion}.

{\bf Case $\tau=0$: }  Estimate \eqref{BoundednessForF} yields a control on $(0,T_{\max})$ only of $I_1(t)=\int_\Omega u \ln (u)$, being necessary a control also on $I_2(t)=\int_\Omega |\nabla v|^2$. Let us deal with this issue. By multiplying the $u$-equation in \eqref{eq:chem-gompertz} by $u$, integrations by parts give  
     \[
     \dfrac{1}{2}\dfrac{d}{dt}\int_\Omega u^2 + \int_\Omega |\nabla u|^2=-\dfrac{\chi}{2}\int_\Omega u^2\Delta v + \alpha \int_\Omega u\ln\left(\frac{K}{u}\right) \quad \textrm{on } (0,T_{\max}).     
     \]
  In particular, from  $\Delta v=v-u$, we obtain 
  \begin{equation}\label{estimation_lem}
      \dfrac{1}{2}\dfrac{d}{dt}\int_\Omega u^2 + \int_\Omega |\nabla u|^2=\dfrac{\chi}{2}\int_\Omega u^3 -\dfrac{\chi}{2}\int_\Omega u^2v+\alpha \int_\Omega u\ln\left(\frac{K}{u}\right)\leq \dfrac{\chi}{2}\|u\|^3_{L^3(\Omega)}+\alpha \int_\Omega u\ln\left(\frac{K}{u}\right) \quad \textrm{on } (0,T_{\max}).
 \end{equation}
     Now, we estimate $\|u\|^3_{L^3(\Omega)}$ by using the extended result of the Gagliardo--Nirenberg inequality from (\cite{BellomoEtAl}, Lemma 3.3, page 1678)), as well as the boundedness of $\|u\ln (u)\|_{L^1(\Omega)}\leq L$ and \eqref{BoundednessMass}, so to obtain 
     \begin{equation}\label{IneqPrevious}
        \|u\|^3_{L^3(\Omega)}\leq \epsilon \|\nabla u\|^2_{L^2(\Omega)}\|u\ln (u)\|_{L^1(\Omega)}+c(\epsilon)\left(\|u\|^3_{L^1(\Omega)}+1\right) \leq \epsilon \|\nabla u\|^2_{L^2(\Omega)}L+c(\epsilon)\left(M^3+1\right)\quad \textrm{on } (0,T_{\max}).     
     \end{equation}
     By choosing  $\epsilon =1/(\chi L)$ in \eqref{IneqPrevious} and inserting the resulting estimate into \eqref{estimation_lem}  provides, once the relation  $\alpha s\ln\left(\frac{K}{s}\right)\leq \alpha K/e$, $s\geq 0$, and  \eqref{BoundednessMass} 
are taken into account,  some $c_2>0$ with the property that
     \begin{equation}\label{Ineq_ellip}
     \dfrac{d}{dt}\|u\|^2_{L^2(\Omega)}+\|\nabla u\|^2_{L^2(\Omega)}\leq c_2 \quad \textrm{for all } t\in (0,T_{\max}).     
     \end{equation}
On the other hand, a further application of the (already exploited) Gagliardo--Nirenberg and Young's inequalities, supported by the boundedness of the mass in  \eqref{BoundednessMass}, leads for some positive $c_3$ (and recalling $(A+B)^2\leq 2(A^2+B^2)$) to
     \[
     \int_\Omega u^2=\|u\|^2_{L^2(\Omega)}\leq c_{GN}^2\left[\|\nabla u\|_{L^2(\Omega)}^\frac{1}{2}  \|u\|_{L^1(\Omega)}^\frac{1}{2}+c_{GN} \| u\|_{L^1(\Omega)}\right]^2\leq \|\nabla u\|_{L^2(\Omega)}^2+c_3 \quad \textrm{for all } t\in (0,T_{\max}),     
     \]
     which plugged into \eqref{Ineq_ellip} gives
     \[
         \frac{d}{dt}\int_\Omega u^2 \leq c_2+c_3-\int_\Omega u^2 \quad \textrm{on } (0,T_{\max}), \quad \textrm{complemented with } \left(\int_\Omega u^2\right)_{t=0}=\int_\Omega u_0(x)^2dx,
     \]
ensuring that $\|u\|_{L^2(\Omega)}$ is bounded on $(0,T_{\max})$. Therefore, by exploiting $u\in L^\infty((0,T_{\max}); L^2(\Omega))$ in the equation for $v$ in \eqref{eq:chem-gompertz},  results of elliptic regularity imply $\nabla v\in L^\infty((0,T_{\max}); W^{1,2}(\Omega))$;  definitively, $I_1(t)=\int_\Omega u \ln (u)$ and $I_2(t)=b\int_\Omega |\nabla v|^2$ are uniformly bounded on $(0,T_{\max})$ and we have the claim by invoking again Lemma \ref{LemmaBoundendesCriterion}.

\paragraph*{Acknowledgments.}  GV is member of the Gruppo Nazionale per l’Analisi Matematica, la
Probabilità e le loro Applicazioni (GNAMPA) of the Istituto Nazionale di Alta Matematica (INdAM),
and is partially supported by the research projects \textit{Partial Differential Equations and their role in understanding natural phenomena} (2023, CUP F23C25000080007), funded by Fondazione di Sardegna, and \textit{Modelli di reazione-diffusione-trasporto: dall'analisi alle applicazioni } (2025, CUP E53C25002010001), funded by GNAMPA-INdAM. NA is supported by the project I-MAROC ``Artificial Intelligence/Applied Mathematics, Health/Environment: Simulation for Decision Support''

\bibliographystyle{cas-model2-names}


\begin{thebibliography}{14}
\expandafter\ifx\csname natexlab\endcsname\relax\def\natexlab#1{#1}\fi
\providecommand{\url}[1]{\texttt{#1}}
\providecommand{\href}[2]{#2}
\providecommand{\path}[1]{#1}
\providecommand{\DOIprefix}{doi:}
\providecommand{\ArXivprefix}{arXiv:}
\providecommand{\URLprefix}{URL: }
\providecommand{\Pubmedprefix}{pmid:}
\providecommand{\doi}[1]{\href{http://dx.doi.org/#1}{\path{#1}}}
\providecommand{\Pubmed}[1]{\href{pmid:#1}{\path{#1}}}
\providecommand{\bibinfo}[2]{#2}
\ifx\xfnm\relax \def\xfnm[#1]{\unskip,\space#1}\fi
\bibitem[{Baghaei et~al.(2026)Baghaei, Frassu, Tanaka and
  Viglialoro}]{BaghaeiEtAl2026JDE}
\bibinfo{author}{Baghaei, K.}, \bibinfo{author}{Frassu, S.},
  \bibinfo{author}{Tanaka, Y.}, \bibinfo{author}{Viglialoro, G.},
  \bibinfo{year}{2026}.
\newblock \bibinfo{title}{To what extent does the consideration of positive
  total flux influence the dynamics of {K}eller--{S}egel-type models?}
\newblock \bibinfo{journal}{J. Differential Equations} \bibinfo{volume}{452},
  \bibinfo{pages}{113808}.
\bibitem[{Bellomo et~al.(2015)Bellomo, Bellouquid, Tao and
  Winkler}]{BellomoEtAl}
\bibinfo{author}{Bellomo, N.}, \bibinfo{author}{Bellouquid, A.},
  \bibinfo{author}{Tao, Y.}, \bibinfo{author}{Winkler, M.},
  \bibinfo{year}{2015}.
\newblock \bibinfo{title}{{Toward a mathematical theory of Keller--Segel models
  of pattern formation in biological tissues}}.
\newblock \bibinfo{journal}{Math. Models Methods Appl. Sci.}
  \bibinfo{volume}{25}, \bibinfo{pages}{1663--1763}.
\bibitem[{Brezis(2011)}]{BrezisBook}
\bibinfo{author}{Brezis, H.}, \bibinfo{year}{2011}.
\newblock \bibinfo{title}{Functional Analysis, {S}obolev Spaces and Partial
  Differential Equations}. volume \bibinfo{volume}{Universitext}.
\newblock \bibinfo{publisher}{Springer-Verlag, New York}.
\bibitem[{Dhar and Bhattacharya(2018)}]{dhar2018comparison}
\bibinfo{author}{Dhar, M.}, \bibinfo{author}{Bhattacharya, P.},
  \bibinfo{year}{2018}.
\newblock \bibinfo{title}{Comparison of the logistic and the {Gompertz} curve
  under different constraints}.
\newblock \bibinfo{journal}{J. Stat. Manag. Syst.} \bibinfo{volume}{21},
  \bibinfo{pages}{1189--1210}.
\bibitem[{Keller and Segel(1970)}]{keller1970initiation}
\bibinfo{author}{Keller, E.F.}, \bibinfo{author}{Segel, L.A.},
  \bibinfo{year}{1970}.
\newblock \bibinfo{title}{Initiation of slime mold aggregation viewed as an
  instability}.
\newblock \bibinfo{journal}{J. Theor. Biol.} \bibinfo{volume}{26},
  \bibinfo{pages}{399--415}.
\bibitem[{Keller and Segel(1971)}]{keller1971model}
\bibinfo{author}{Keller, E.F.}, \bibinfo{author}{Segel, L.A.},
  \bibinfo{year}{1971}.
\newblock \bibinfo{title}{Model for chemotaxis}.
\newblock \bibinfo{journal}{J. Theor. Biol.} \bibinfo{volume}{30},
  \bibinfo{pages}{225--234}.
\bibitem[{Lady\v{z}enskaja et~al.(1988)Lady\v{z}enskaja, Solonnikov and
  Ural'ceva}]{LSUBookInequality}
\bibinfo{author}{Lady\v{z}enskaja, O.A.}, \bibinfo{author}{Solonnikov, V.A.},
  \bibinfo{author}{Ural'ceva, N.N.}, \bibinfo{year}{1988}.
\newblock \bibinfo{title}{{Linear and Quasi-Linear Equations of Parabolic
  Type}}, in: \bibinfo{booktitle}{{Translations of Mathematical Monographs}}.
  \bibinfo{publisher}{American Mathematical Society}.
  volume~\bibinfo{volume}{23}.
\bibitem[{Osaki et~al.(2002)Osaki, Tsujikawa, Yagi and
  Mimura}]{osaki2002exponential}
\bibinfo{author}{Osaki, K.}, \bibinfo{author}{Tsujikawa, T.},
  \bibinfo{author}{Yagi, A.}, \bibinfo{author}{Mimura, M.},
  \bibinfo{year}{2002}.
\newblock \bibinfo{title}{Exponential attractor for a chemotaxis-growth system
  of equations}.
\newblock \bibinfo{journal}{Nonlinear Anal. Theory Methods Appl.}
  \bibinfo{volume}{51}, \bibinfo{pages}{119--144}.
\bibitem[{Tello and Winkler(2007)}]{TelloWinkParEl}
\bibinfo{author}{Tello, J.I.}, \bibinfo{author}{Winkler, M.},
  \bibinfo{year}{2007}.
\newblock \bibinfo{title}{A chemotaxis system with logistic source}.
\newblock \bibinfo{journal}{Comm. Partial Differential Equations}
  \bibinfo{volume}{32}, \bibinfo{pages}{849--877}.
\bibitem[{Vaghi et~al.(2020)Vaghi, Rodallec, Fanciullino, Ciccolini, Mochel,
  Mastri, Poignard, Ebos and Benzekry}]{vaghi2020population}
\bibinfo{author}{Vaghi, C.}, \bibinfo{author}{Rodallec, A.},
  \bibinfo{author}{Fanciullino, R.}, \bibinfo{author}{Ciccolini, J.},
  \bibinfo{author}{Mochel, J.P.}, \bibinfo{author}{Mastri, M.},
  \bibinfo{author}{Poignard, C.}, \bibinfo{author}{Ebos, J.M.L.},
  \bibinfo{author}{Benzekry, S.}, \bibinfo{year}{2020}.
\newblock \bibinfo{title}{Population modeling of tumor growth curves and the
  reduced {Gompertz} model improve prediction of the age of experimental
  tumors}.
\newblock \bibinfo{journal}{PLOS Comput. Biol} \bibinfo{volume}{16},
  \bibinfo{pages}{e1007178}.
\bibitem[{Winkler(2010)}]{winkler2010boundedness}
\bibinfo{author}{Winkler, M.}, \bibinfo{year}{2010}.
\newblock \bibinfo{title}{Boundedness in the higher-dimensional
  parabolic-parabolic chemotaxis system with logistic source}.
\newblock \bibinfo{journal}{Commun. Partial Differ. Equ.} \bibinfo{volume}{35},
  \bibinfo{pages}{1516--1537}.
\bibitem[{Winkler(2011)}]{Winkler2011-BlowUHighDimensionLogistic}
\bibinfo{author}{Winkler, M.}, \bibinfo{year}{2011}.
\newblock \bibinfo{title}{Blow-up in a higher-dimensional chemotaxis system
  despite logistic growth restriction}.
\newblock \bibinfo{journal}{J. Math. Anal. Appl.} \bibinfo{volume}{384},
  \bibinfo{pages}{261--272}.
\bibitem[{Xiang(2018)}]{xiang2018sub}
\bibinfo{author}{Xiang, T.}, \bibinfo{year}{2018}.
\newblock \bibinfo{title}{Sub-logistic source can prevent blow-up in the 2d
  minimal {Keller--Segel} chemotaxis system}.
\newblock \bibinfo{journal}{J. Math. Phys.} \bibinfo{volume}{59}.
\bibitem[{Zheng et~al.(2018)Zheng, Li, Bao and Zou}]{zheng2018new}
\bibinfo{author}{Zheng, J.}, \bibinfo{author}{Li, Y.}, \bibinfo{author}{Bao,
  G.}, \bibinfo{author}{Zou, X.}, \bibinfo{year}{2018}.
\newblock \bibinfo{title}{A new result for global existence and boundedness of
  solutions to a parabolic--parabolic {Keller--Segel} system with logistic
  source}.
\newblock \bibinfo{journal}{J. Math. Anal. Appl.} \bibinfo{volume}{462},
  \bibinfo{pages}{1--25}.

\end{thebibliography}

\end{document}